\documentclass[12pt]{article}

\usepackage{amsfonts}
\usepackage{amsmath}
\usepackage{amsthm}
\usepackage{amssymb}

\usepackage{xcolor}
\usepackage{hyperref}
\usepackage[nottoc]{tocbibind} 

\usepackage{mathtools}
\usepackage{longtable}

\allowdisplaybreaks

\hypersetup{
pdftitle={When pi(n) does not divide n},
pdfauthor={Germ\'an Andr\'es Paz},
pdfsubject={},
pdfkeywords={bounds on the prime-counting function, explicit formulas for the prime-counting function, intervals, prime numbers, sequences}
}

\hypersetup{colorlinks=true,linkcolor=blue,citecolor=green,urlcolor=orange}

\makeatletter
\def\thmhead@plain#1#2#3{%
  \thmname{#1}\thmnumber{\@ifnotempty{#1}{ }\@upn{#2}}%
  \thmnote{ {\the\thm@notefont#3}}}
\let\thmhead\thmhead@plain
\makeatother

\newcommand{\xqedhere}[2]{%
  \rlap{\hbox to#1{\hfil\llap{\ensuremath{#2}}}}}

\theoremstyle{definition}

\newtheorem{Theorem}{Theorem}[section]
\newtheorem{Remark}[Theorem]{Remark}
\newtheorem{Example}[Theorem]{Example}
\newtheorem{Question}[Theorem]{Question}
\newtheorem{Corollary}[Theorem]{Corollary}

\renewcommand*{\qed}{\hfill\ensuremath{\blacksquare}}
\renewcommand\footnotemark{}

\begin{document}

\title{When $\pi(n)$ does not divide $n$}
\author{Germ\'an Andr\'es Paz}
\date{January 26, 2015}

\maketitle

\begin{abstract}
Let $\pi(n)$ denote the prime-counting function and let
\begin{equation*}
f(n)=\left|\left\lfloor\log n-\lfloor\log n\rfloor-0.1\right\rfloor\right|\left\lfloor\frac{\left\lfloor n/\lfloor\log n-1\rfloor\right\rfloor\lfloor\log n-1\rfloor}{n}\right\rfloor\text{.}
\end{equation*}
In this paper we prove that if $n$ is an integer $\ge 60184$ and $f(n)=0$, then $\pi(n)$ does not divide $n$. We also show that if $n\ge 60184$ and $\pi(n)$ divides $n$, then $f(n)=1$. In addition, we prove that if $n\ge 60184$ and $n/\pi(n)$ is an integer, then $n$ is a multiple of $\lfloor\log n-1\rfloor$ located in the interval $[e^{\lfloor\log n-1\rfloor+1},e^{\lfloor\log n-1\rfloor+1.1}]$. This allows us to show that if $c$ is any fixed integer $\ge 12$, then in the interval $[e^c,e^{c+0.1}]$ there is always an integer $n$ such that $\pi(n)$ divides $n$.

Let $S$ denote the sequence of integers generated by the function $d(n)=n/\pi(n)$ (where $n\in\mathbb{Z}$ and $n>1$) and let $S_k$ denote the $k$th term of sequence $S$. Here we ask the question whether there are infinitely many positive integers $k$ such that $S_k=S_{k+1}$.
\end{abstract}



\section{Notation}

Throughout this paper the number $n$ is always a positive integer. Moreover, we use the following notation:

\begin{itemize}

\item $|\cdot|$ (absolute value)

\item $\lceil\cdot\rceil$ (ceiling function)

\item $\mid$ (divides)

\item $\nmid$ (does not divide)

\item $\lfloor\cdot\rfloor$ (floor function)

\item $\operatorname{frac}(\cdot)$ (fractional part)

\item $\log$ (natural logarithm)

\end{itemize}

\section{Introduction}

Determining how prime numbers are distributed among natural numbers is one of the most difficult mathematical problems. This explains why the prime-counting function $\pi(n)$ (which counts the number of primes less than or equal to a given number $n$) has been one of the main objects of study in Mathematics for centuries.

In \cite{Gaitanas} Gaitanas obtains an explicit formula for $\pi(n)$ that holds infinitely often. His proof is based on the fact that the function $d(n)=n/\pi(n)$ takes on every integer value greater than 1 (as proved by Golomb \cite{Golomb}) and on the fact that $x/(\log x-0.5)<\pi(x)<x/(\log x-1.5)$ for $x\ge 67$ (as shown by Rosser and Schoenfeld \cite{Rosser_Schoenfeld}). In this paper we find alternative expressions that are valid for infinitely many positive integers $n$, and we also prove, among other results, that if $n\ge 60184$ and
\begin{equation*}
\left|\left\lfloor\log n-\lfloor\log n\rfloor-0.1\right\rfloor\right|\left\lfloor\frac{\left\lfloor n/\lfloor\log n-1\rfloor\right\rfloor\lfloor\log n-1\rfloor}{n}\right\rfloor
\end{equation*}
equals 0, then $\pi(n)$ does not divide $n$.

We will place emphasis on the following three theorems, which were proved by Golomb, Dusart, and Gaitanas respectively:

\begin{Theorem}[\cite{Golomb}]\label{theorem0}
The function $d(n)=n/\pi(n)$ takes on every integer value greater than 1.\qed
\end{Theorem}

\begin{Theorem}[\cite{Dusart}]\label{theorem1}
If $n$ is an integer $\ge 60184$, then
\begin{equation*}
\frac{n}{\log n-1}<\pi(n)<\frac{n}{\log n-1.1}\text{.}\xqedhere{4.01cm}{\qed}
\end{equation*}
\end{Theorem}

\begin{Remark}
Dusart's paper states that for $x\ge 60184$ we have $x/(\log x-1)\le \pi(x)\le x/(\log x-1.1)$, but since $\log n$ is always irrational when $n$ is an integer $>1$, we can state his theorem the way we did.\hfill$\blacktriangleleft$
\end{Remark}

\begin{Theorem}[\cite{Gaitanas}]\label{theorem2}
The formula
\begin{equation*}
\pi(n)=\frac{n}{\lfloor\log n-0.5\rfloor}
\end{equation*}
is valid for infinitely many positive integers $n$.\qed
\end{Theorem}

\section{Main results}

We are now ready to prove our main results:

\begin{Theorem}\label{theorem3}
The formula
\begin{equation*}
\pi(n)=\frac{n}{\lfloor\log n-1\rfloor}
\end{equation*}
holds for infinitely many positive integers $n$.\qed
\end{Theorem}

\begin{proof}
According to Theorem \ref{theorem1}, for $n\ge 60184$ we have
\begin{equation*}
\frac{n}{\log n-1}<\pi(n)<\frac{n}{\log n-1.1}\Rightarrow\frac{\log n-1.1}{n}<\frac{1}{\pi(n)}<\frac{\log n-1}{n}\text{.}
\end{equation*}
If we multiply by $n$, we get
\begin{equation}\label{equation1}
\log n-1.1<\frac{n}{\pi(n)}<\log n-1\text{.}
\end{equation}
Since $\log n-1.1$ and $\log n-1$ are both irrational (for $n>1$), inequality \eqref{equation1} implies that when $n/\pi(n)$ is an integer we must have
\begin{equation}\label{equation2}
\frac{n}{\pi(n)}=\lfloor\log n-1\rfloor=\lfloor\log n-1.1\rfloor+1=\lceil\log n-1.1\rceil=\lceil\log n-1\rceil-1\text{.} 
\end{equation}
Taking Theorem \ref{theorem1} and equality \eqref{equation2} into account, we can say that for every $n\ge 60184$ when $n/\pi(n)$ is an integer we must have
\begin{equation*}
\frac{n}{\pi(n)}=\lfloor\log n-1\rfloor\Rightarrow\pi(n)=\frac{n}{\lfloor\log n-1\rfloor}\text{.}
\end{equation*}
Since Theorem \ref{theorem0} implies that $n/\pi(n)$ is an integer infinitely often, it follows that there are infinitely many positive integers $n$ such that $\pi(n)=n/\lfloor\log n-1\rfloor$.
\end{proof}

In fact, the following theorem follows from Theorems \ref{theorem0}, from Gaitana's proof of Theorem \ref{theorem2}, and from the proof of Theorem \ref{theorem3}:

\begin{Theorem}\label{theorem4}
For every $n\ge 60184$ when $n/\pi(n)$ is an integer we must have
\begin{equation}\label{equation3}
\begin{gathered}
\frac{n}{\pi(n)}=\lceil\log n-1.5\rceil=\lfloor\log n-0.5\rfloor=\lfloor\log n-1\rfloor=\\=\lfloor\log n-1.1\rfloor+1=\lceil\log n-1.1\rceil=\lceil\log n-1\rceil-1\text{.}
\end{gathered}
\end{equation}
In other words, for $n\ge 60184$ when $n/\pi(n)$ is an integer we must have
\begin{equation*}
\begin{gathered}
\pi(n)=\frac{n}{\lceil\log n-1.5\rceil}=\frac{n}{\lfloor\log n-0.5\rfloor}=\frac{n}{\lfloor\log n-1\rfloor}=\frac{n}{\lfloor\log n-1.1\rfloor+1}=\\=\frac{n}{\lceil\log n-1.1\rceil}=\frac{n}{\lceil\log n-1\rceil-1}\text{.}\xqedhere{3.71cm}{\qed} 
\end{gathered}
\end{equation*}
\end{Theorem}

\begin{Theorem}\label{theorem5}
Let $n$ be an integer $\ge 60184$. If $\operatorname{frac}(\log n)=\log n-\lfloor\log n\rfloor>0.1$, then $\pi(n)\nmid n$ (that is to say, $n/\pi(n)$ is not an integer).\qed
\end{Theorem}

\begin{proof}
According to Theorem \ref{theorem4}, if $n\ge 60184$ and $n/\pi(n)$ is an integer, then
\begin{align*}
\frac{n}{\pi(n)}=\lfloor\log n-1\rfloor&=\lceil\log n-1.1\rceil\text{.}
\intertext{In other words, for $n\ge 60184$ when $n/\pi(n)$ is an integer we have}
\lfloor\log n-1\rfloor&=\lceil\log n-1.1\rceil\\
\lfloor\log n-1\rfloor&=\lceil\log n-1-0.1\rceil\\
\operatorname{frac}(\log n-1)&\le 0.1\\
\log n-1-\lfloor\log n-1\rfloor&\le 0.1\\
\log n-\lfloor\log n-1\rfloor&\le 1.1\\
\operatorname{frac}(\log n)&\le 0.1\\
\log n-\lfloor\log n\rfloor&\le 0.1\text{.}
\end{align*}
Suppose that $P$ is the statement `$n/\pi(n)$ is an integer' and $Q$ is the statement `$\log n-\lfloor\log n\rfloor\le 0.1$'. According to propositional logic, the fact that $P\rightarrow Q$ implies that $\neg Q\rightarrow\neg P$.
\end{proof}

Similar theorems can be proved by using Theorem \ref{theorem4} and equality \eqref{equation3}.

\begin{Remark}
We can also say that if $n\ge 60184$ and
\begin{equation*}
n>e^{0.1+\lfloor\log n\rfloor}\text{,}
\end{equation*}
then $\pi(n)\nmid n$.\hfill$\blacktriangleleft$
\end{Remark}

\begin{Remark}
Because $\log n$ is irrational for $n>1$, another way of stating Theorem \ref{theorem5} is by saying that if $n\ge 60184$ and the first digit to the right of the decimal point of $\log n$ is 1, 2, 3, 4, 5, 6, 7, 8, or 9, then $\pi(n)\nmid n$. Example:
\begin{equation*}
\log 10^{31}=71.\textcolor{red}{3}8...
\end{equation*}
The first digit after the decimal point of $\log 10^{31}$ (in red) is 3. This implies that $\pi(10^{31})$ does not divide $10^{31}$. We can also say that if $n\ge 60184$ and $\pi(n)$ divides $n$, then the first digit after the decimal point of $\log n$ can only be 0.

Now, if $y$ is a positive noninteger, then the first digit after the decimal point of $y$ is equal to $\lfloor 10\operatorname{frac}(y)\rfloor=\left\lfloor 10y-10\lfloor y\rfloor\right\rfloor$. So, we can say that if $n\ge 60184$ and $\left\lfloor 10\log n-10\lfloor\log n\rfloor\right\rfloor\ne 0$, then $\pi(n)\nmid n$. On the other hand, if $n\ge 60184$ and $\pi(n)$ divides $n$, then $\left\lfloor 10\log n-10\lfloor\log n\rfloor\right\rfloor= 0$.\hfill$\blacktriangleleft$
\end{Remark}

The following theorem follows from Theorem \ref{theorem5}:

\begin{Theorem}\label{theorem6}
Let $e$ be the base of the natural logarithm. If $a$ is any integer $\ge 11$ and $n$ is any integer contained in the interval $[e^{a+0.1},e^{a+1}]$, then $\pi(n)\nmid n$. (The number $e^r$ is irrational when $r$ is a rational number $\ne 0$.)\qed
\end{Theorem}

\begin{Example}
Take $a=18$. If $n$ is any integer in the interval $[e^{18.1},e^{19}]$, then $\pi(n)\nmid n$.\hfill$\blacktriangleleft$
\end{Example}

\begin{Corollary}
If $a$ is any positive integer $>1$, then $\pi(\lfloor e^a\rfloor)\nmid \lfloor e^a\rfloor$.\qed
\end{Corollary}

\begin{proof}
For $a\ge 12$ the proof follows from Theorem \ref{theorem6}. On the other hand, $\lfloor e^a\rfloor/\pi(\lfloor e^a\rfloor)$ is not an integer whenever $2\le a\le 11$, as shown in the following table:

\begin{longtable}{|c|c|}
\hline
   $a$ & $\lfloor e^a\rfloor/\pi(\lfloor e^a\rfloor)$  \\
\hline \hline
\endhead
\multicolumn{2}{c}{Continued on next page...}
\endfoot
\endlastfoot
\hline
   {\bf 1} & 2\textcolor{white}{.00...}\\
\hline
   {\bf 2} & 1.75\textcolor{white}{...}\\
\hline
   {\bf 3} & 2.5\textcolor{white}{0...}\\
\hline
   {\bf 4} & 3.37...\\
\hline
   {\bf 5} & 4.35...\\
\hline
{\bf 6} & 5.10...\\
\hline
{\bf 7} & 5.98...\\
\hline
{\bf 8} & 6.94...\\
\hline
{\bf 9} & 7.95...\\
\hline
{\bf 10} & 8.93...\\
\hline
{\bf 11} & 9.89...\\
\hline
\end{longtable}
In other words, if $a\in\mathbb{Z}^+$, then $\pi(\lfloor e^a\rfloor)\mid\lfloor e^a\rfloor$ only when $a=1$.
\end{proof}

\begin{Theorem}\label{theorem8}
Let $n$ be an integer $\ge 60184$ and let
\begin{equation*}
f(n)=\left|\left\lfloor\log n-\lfloor\log n\rfloor-0.1\right\rfloor\right|\left\lfloor\frac{\left\lfloor n/\lfloor\log n-1\rfloor\right\rfloor\lfloor\log n-1\rfloor}{n}\right\rfloor\text{.}
\end{equation*}
If $f(n)=0$, then $\pi(n)\nmid n$. On the other hand, if $\pi(n)\mid n$, then $f(n)=1$.\qed
\end{Theorem}

\begin{proof}

~\vspace{\topsep}

$\bullet$ {\bf Part 1}

\vspace{\topsep}

\noindent Suppose that
\begin{equation*}
f(n)=g(n)h(n)\text{,}
\end{equation*}
where
\begin{equation*}
g(n)=\left|\left\lfloor\log n-\lfloor\log n\rfloor-0.1\right\rfloor\right|
\end{equation*}
and
\begin{equation*}
h(n)=\left\lfloor\frac{\left\lfloor n/\lfloor\log n-1\rfloor\right\rfloor\lfloor\log n-1\rfloor}{n}\right\rfloor\text{.}
\end{equation*}

To begin with, if $n\ge 60184$, then $\log n-\lfloor\log n\rfloor$ can never be equal to 0.1. Now, when $\log n-\lfloor\log n\rfloor<0.1$ we have $-1<\log n-\lfloor\log n\rfloor-0.1<0$ and hence $\left|\left\lfloor\log n-\lfloor\log n\rfloor-0.1\right\rfloor\right|=1$. On the other hand, when $\log n-\lfloor\log n\rfloor>0.1$ we have $0<\log n-\lfloor\log n\rfloor-0.1<1$ and hence $\left|\left\lfloor\log n-\lfloor\log n\rfloor-0.1\right\rfloor\right|=0$. This means that if $n$ is any integer $\ge 60184$, then $g(n)$ equals either 0 or 1. We can also say that if $n\ge 60184$ and $g(n)=0$, then $\log n-\lfloor\log n\rfloor>0.1$, which implies that $\pi(n)\nmid n$ (according to Theorem \ref{theorem5}). (This means that if $n\ge 60184$ and $\pi(n)\mid n$, then $g(n)=1$.)

\vspace{\topsep}

$\bullet$ {\bf Part 2}

\vspace{\topsep}

\noindent If $n\ge 60184$, then
\begin{equation*}
\left\lfloor\frac{n}{\lfloor\log n-1\rfloor}\right\rfloor\le\frac{n}{\lfloor\log n-1\rfloor}\text{,}
\end{equation*}
which means that
\begin{equation*}
\left\lfloor\left\lfloor\frac{n}{\lfloor\log n-1\rfloor}\right\rfloor/\frac{n}{\lfloor\log n-1\rfloor}\right\rfloor=\left\lfloor\frac{\left\lfloor n/\lfloor\log n-1\rfloor\right\rfloor\lfloor\log n-1\rfloor}{n}\right\rfloor=h(n)
\end{equation*}
equals either 0 or 1. If $h(n)=0$, then $n$ is not divisible by $\lfloor\log n-1\rfloor$, which implies that $\pi(n)\nmid n$ (according to Theorem \ref{theorem4}). In other words, if $n\ge 60184$ and $h(n)=0$, then $\pi(n)\nmid n$. (This means that if $n\ge 60184$ and $\pi(n)\mid n$, then $h(n)=1$.)

\vspace{\topsep}

$\bullet$ {\bf Part 3}

\vspace{\topsep}

\noindent There are two possible outputs for $g(n)$ (0 or 1) as well as two possible outputs for $h(n)$ (0 or 1). This means that for $n\ge 60184$ we have either
\begin{align*}
g(n)h(n)=0\cdot 0=0\text{,}\\
\shortintertext{or}
g(n)h(n)=0\cdot 1=0\text{,}\\
\shortintertext{or}
g(n)h(n)=1\cdot 0=0\text{,}\\
\shortintertext{or}
g(n)h(n)=1\cdot 1=1\text{.}
\end{align*}

If $f(n)=g(n)h(n)=0$, then at least one of the factors $g(n)$ and $h(n)$ equals 0, which implies that $\pi(n)\nmid n$ (see Part 1 and Part 2). This means that if $n\ge 60184$ and $f(n)=0$, then $\pi(n)\nmid n$. Consequently, if $n\ge 60184$ and $\pi(n)\mid n$, then $f(n)=1$.
\end{proof}

\begin{Theorem}\label{theorem7}
If $n\ge 60184$ and $n/\pi(n)$ is an integer, then $n$ is a multiple of $\lfloor\log n-1\rfloor$ located in the interval $[e^{\lfloor\log n-1\rfloor+1},e^{\lfloor\log n-1\rfloor+1.1}]$.\qed
\end{Theorem}

\begin{proof}\let\qed\relax
According to Theorems \ref{theorem4} and \ref{theorem5}, if $n\ge 60184$ and $n/\pi(n)$ is an integer, then
\begin{equation*}
\frac{n}{\pi(n)}=\lfloor\log n-1\rfloor\Rightarrow n=\pi(n)\lfloor\log n-1\rfloor
\end{equation*}
and
\begin{equation*}
\text{frac}(\log n)=\log n-\lfloor\log n\rfloor\le 0.1\text{.}
\end{equation*}
The fact that frac$(\log n)\le 0.1$ implies that $n$ is located in the interval
\begin{equation*}
[e^k,e^{k+0.1}]
\end{equation*}
for some positive integer $k$. In other words, we have
\begin{equation*}
e^k<n<e^{k+0.1}\Rightarrow k<\log n<k+0.1\Rightarrow k-1<\log n-1<k-0.9\text{,}
\end{equation*}
which means that
\begin{align*}
k-1&=\lfloor\log n-1\rfloor\\
k&=\lfloor\log n-1\rfloor+1\text{.}\xqedhere{4.71cm}{\blacksquare}
\end{align*}
\end{proof}

\begin{Remark}\label{theorem9}
Suppose that $b$ is any fixed integer $\ge 12$. Theorem \ref{theorem7} implies that if $n$ is an integer in the interval $[e^b,e^{b+0.1}]$ and at the same time $n$ is not a multiple of $b-1$, then $\pi(n)\nmid n$. This means that if $n\ge 60184$ and $\pi(n)$ divides $n$, then $n$ is located in the interval $[e^b,e^{b+0.1}]$ for some positive integer $b$ and $n$ is a multiple of $b-1$.\hfill$\blacktriangleleft$
\end{Remark}

The following theorem follows from Theorems \ref{theorem0} and \ref{theorem7} and from the fact that $n/\pi(n)<11$ for $n\le 60183$ (this fact can be checked using software):

\begin{Theorem}\label{theorem10}
Let $c$ be any fixed integer $\ge 12$. In the interval $[e^c,e^{c+0.1}]$ there is always an integer $n$ such that $\pi(n)$ divides $n$. In other words, in the interval $[e^c,e^{c+0.1}]$ there is always an integer $n$ such that $\pi(n)=n/(c-1)$.\qed
\end{Theorem}

\section{Conclusion and Further Discussion}

The following are the main theorems of this paper:

\vspace{\topsep}

\noindent {\bf Theorem \ref{theorem8}.} Let $n$ be an integer $\ge 60184$ and let
\begin{equation*}
f(n)=\left|\left\lfloor\log n-\lfloor\log n\rfloor-0.1\right\rfloor\right|\left\lfloor\frac{\left\lfloor n/\lfloor\log n-1\rfloor\right\rfloor\lfloor\log n-1\rfloor}{n}\right\rfloor\text{.}
\end{equation*}
If $f(n)=0$, then $\pi(n)\nmid n$. On the other hand, if $\pi(n)\mid n$, then $f(n)=1$.\qed

\vspace{\topsep}

\noindent {\bf Theorem \ref{theorem7}.} If $n\ge 60184$ and $n/\pi(n)$ is an integer, then $n$ is a multiple of $\lfloor\log n-1\rfloor$ located in the interval $[e^{\lfloor\log n-1\rfloor+1},e^{\lfloor\log n-1\rfloor+1.1}]$.\qed

\vspace{\topsep}

\noindent {\bf Theorem \ref{theorem10}.} Let $c$ be any fixed integer $\ge 12$. In the interval $[e^c,e^{c+0.1}]$ there is always an integer $n$ such that $\pi(n)$ divides $n$. In other words, in the interval $[e^c,e^{c+0.1}]$ there is always an integer $n$ such that $\pi(n)=n/(c-1)$.\qed

\vspace{\topsep}

We recall that Golomb \cite{Golomb} proved that for every integer $n>1$ there exists a positive integer $m$ such that $m/\pi(m)=n$. Suppose now that $R$ is the sequence of numbers generated by the function $d(n)=n/\pi(n)$ ($n\in\mathbb{Z}$ and $n>1$). In other words,
\begin{equation*}
R=(2,\quad1.5,\quad2,\quad1.66...,\quad2,\quad1.75,\quad2,\quad2.25,\quad2.5,\quad\dots)\text{.}
\end{equation*}
Suppose also that $S$ is the sequence of \emph{integers} generated by the function $d(n)=n/\pi(n)$. In other words,
\begin{equation*}
S=(2,\quad2,\quad2,\quad2,\quad3,\quad3,\quad3,\quad4,\quad4,\quad\dots)\text{.}
\end{equation*}
Motivated by Golomb's result and Theorem \ref{theorem10} we ask the following question:

\begin{Question}\label{question1}
Are there infinitely many positive integers $a$ such that in the interval $[e^a,e^{a+0.1}]$ there are at least two distinct positive integers $n_1$ and $n_2$ such that $\pi(n_1)\mid n_1$ and $\pi(n_2)\mid n_2$? In other words, are there infinitely many positive integers $n$ that can be expressed as $m/\pi(m)$ in more than one way?\hfill$\blacktriangleleft$
\end{Question}

Now, let $S_k$ denote the $k$th term of sequence $S$. Clearly, Question \ref{question1} is equivalent to the following question:

\begin{Question}
Are there infinitely many positive integers $k$ such that $S_k=S_{k+1}$?\hfill$\blacktriangleleft$
\end{Question}


\noindent\emph{Instituto de Educaci\'on Superior N$^\circ$28 Olga Cossettini, (2000) Rosario, Santa Fe, Argentina\\E-mail: \href{mailto:germanpaz\_ar@hotmail.com}{germanpaz\_ar@hotmail.com}}


\begin{thebibliography}{HD}

\bibitem[1]{Dusart}
Dusart, P. ``Estimates of Some Functions Over Primes without R.H." arXiv:1002.0442 [math.NT], 2010.

\bibitem[2]{Gaitanas}
Gaitanas, K. N. ``An explicit formula for the prime counting function." arXiv:1311.1398 [math.NT], 2013.

\bibitem[3]{Golomb}
Golomb, S. W. ``On the Ratio of $N$ to $\pi(N)$." \emph{The American Mathematical Monthly}. Vol. 69, No. 1, pp. 36--37, 1962.

\bibitem[4]{Rosser_Schoenfeld}
Rosser, J. B.; Schoenfeld, L. ``Approximate formulas for some functions of prime numbers.'' \emph{Illinois Journal of Mathematics}. Vol. 6, No. 1, pp. 64--94, 1962.

\end{thebibliography}
\end{document}